\documentclass{amsart}  	
\usepackage{geometry}                		
\usepackage{graphicx}				
\usepackage{amsthm}								
\usepackage{amssymb}
\newtheorem{theorem}{Theorem}[section]

\newtheorem{lemma}[theorem]{Lemma}
\newtheorem{corollary}[theorem]{Corollary}

\newtheorem{note}{Note}

\begin{document}

\title{Betti numbers of split graphs}
\author{Ralf Fr\"oberg}
\footnote{Department of mathematics, Stockholm university, Sweden, email: frobergralf@gmail.com, ORCID: 0000-0002-7294-2852}
\date{}							

\begin{abstract}
A split graph  is a graph where the vertices are a disjoint union of a complete part $C=\{x_i,\ldots,x_n\}$ and a stable part $S=\{y_1,\ldots,y_m\}$.
We will determine the Betti numbers of the edge ring of all split graphs, in particular show that the only nonzero Betti numbers are $\beta_{0,0}$
and $\beta_{i,i+1}$, $i>0$. The Betti numbers only depend on the multiset of the number of neighbors in $S$ the $x_i$'s have. Singh
and Verma have earlier determined the Betti numbers for complete split graphs (where all $y_i$ are neighbors to all $x_j$), and
for "nearly complete" split graphs (where all $y_i$ are neighbors to all $x_j$, except that $y_i$ is not a neighbor to $x_i$ for $i=1,\ldots,\min\{m,n\}$).
We also determine which split graphs that have Cohen-Macaulay edge ring.
\end{abstract}

\maketitle
\section{Preliminaries}
\subsection{Graphs}
By a graph $G=(V,E)$ we mean a finite simple graph. Thus $V$ is a finite set of vertices $V=\{v_1,\ldots,v_n\}$ and $E$ consists
of edges $(v_i,v_j)$, $i\ne j$, at most one for each pair of vertices. $G_1=(V_1,E_1)$ is a subgraph of $G_2=(V_2,E_2)$ if
$V_1\subseteq V_2$ and $E_1\subseteq E_2$. The subgraph is induced if $E_1=\{(v_i,v_j);v_1,v_2\in V_1,/v_1,v_2)\in E_2\}$.
The complement $\overline G=(V,\overline E)$ of $G$  has the same
vertices as $G$ and $(v_,v_j)\in\overline E$ if and only if $(v_i,v_j)\notin E$. A graph with $V=(\{v_1,\ldots,v_n\},E)$ is called complete, 
and is denoted $K_n$, if $(v_i,v_j)\in E$
for all $i\ne j$. A subset $S\subseteq V$ of a graph $G=(V,E)$ is called stable if there are no edges in $G$ with both vertices in $S$. A graph is
chordal if there are no induced cycles of length $>3$ in $G$, it is cochordal if $\overline G$ is chordal.

The edge ring $k[G]$ over a field $k$, of a graph $G=(\{v_1,\ldots,v_n\},E)$ is $k[x_1,\ldots,x_n]/(x_ix_j;(v_i,v_j)\in E)$.

\subsection{Simplicial complexes}
A (finite) simplicial complex is a set $\Delta$ consisting of subset (simplices), called faces, of a finite set $V=\{v_1,\ldots,v_n\}$ of vertices, such 
that $\{v_i\}\in\Delta$ for all $i$, and such that if $S\in\Delta$ and $T\subseteq S$, then $T\in\Delta$. The dimension of a face is one less
than its number of vertices.

For a simplicial complex $\Delta$ on $\{v_1,\ldots,v_n\}$, the Stanley-Reisner ring over a field $k$ is $k[x_1,\ldots,x_n]/I$, where $I$ is generated by all
squarefree monomials $x_{i_1}\cdots x_{i_k}$ for all $i_1,\ldots,i_k$ for which $\{v_{i_1},\ldots,v_{i_k}\}$ are not a face of $\Delta$.

We will now define a very special kind of simplicial complexes, fat forests. Fat forests are defined recursively as follows. 
A $d$-simplex (i.e. with $d+1$ vertices) $F_1$ of dimension $\ge0$ is a fat forest.
If $F_i$, $i=1,\ldots,k$, are simplices and $G_{k-1}=F_1\cup\cdots\cup F_{k-1}$ is a fat forest, then $G_{k-1}\cup F_k$ is a fat forest if 
$H=G_{k-1}\cap F_k$ is a simplex, $\dim H\ge -1$. (If $\dim H=-1$, then $G_{k-1}$ and $F_k$ are disjoint.)
It is easy to get the Hilbert series for the Stanley-Reisner ring of a fat forest.

The following is \cite[Theorem 1]{Fr1}
\begin{theorem}
Let $F=F_1\cup\cdots\cup F_k$ be a fat forest with $F_i$ a simplex of dimension $d_i-1$ and $F_1\cup\cdots\cup F_{j-1}\cap F_j$ a
simplex of dimension $r_j-1$. Then the Hilbert series $\sum_{i\ge0}\dim_kk[F]_it^i$ of $k[F]$ is 
$$\sum_{i=1}^k\frac{1}{(1-t)^{d_i}}-\sum_{i=2}^k\frac{1}{(1-t)^{r_i}}.$$
\end{theorem}

\subsection{Betti numbers}
A graded algebra $R=S/I$, where $S=k[x_1,\ldots,x_n]$, has a 2-linear resolution if $I$ is generated in degree 2 and all higher syzygies are linear,
equivalently if ${\rm Tor}_{i,j}^S(R,k)=0$ if $j\ne i+1$ for $i>0$. The graded Betti number $\dim_k({\rm Tor}_{i,j}^S(R,k))$ is denoted $\beta_{i,j}(R)$.
The $i$th total Betti number is $b_i=\sum_j\beta_{i,j}$.
A graded algebra $R=S/I$ has a minimal graded resolution like this:
$$0\leftarrow S/I\leftarrow S\leftarrow \oplus_{j=1}^{b_1}S[-a_{1,j}]\leftarrow \oplus_{j=1}^{b_2}S[-a_{2,j}]\leftarrow\cdots\leftarrow \oplus_{j=1}^{b_p}S[-a_{p,j}]\leftarrow0$$
where $S[-k]$ means that we have shifted $S$ by $k$ steps. Using that the alternating sum of the dimension in an exact sequence of vector spaces
is zero, we get the following well known expression for the Hilbert series $\sum_{i\ge0}\dim_kR_it^i$ of $R$.

\begin{theorem}\label{res}
A graded algebra with a resolution as above has Hilbert series $$\frac{\sum_{i=0}^p(-1)^i\sum_{j=1}^{b_i}t^{a_{i,j}}}{(1-t)^n}.$$
\end{theorem}

 Now
$S/I$ has an $2$-linear resolution if for each $i$ we have $a_{i,j}=i+1$ for all $j$. 
are linear. Thus we have

\begin{lemma}\label{lin} A ring $k[x_1,\ldots,x_n]/I$ with $2$-linear resolution
has Hilbert series 
$$\frac{1-b_1t^{2}+b_2t^{3}-\cdots+(-1)^pb_pt^{p+1}}{(1-t)^n}.$$
\end{lemma}

So, if $S/I$ has an $2$-linear resolution, the Hilbert series gives the graded Betti numbers. 
For a graph $G$, let $\Delta(G)$ be the largest simplicial complex with $G$ as 1-skeleton.
The following theorems summarize results from \cite{Fr} and \cite{Fr1}.

\begin{theorem}
If $G$ is a graph such that $k[G]$ has a 2-linear resolution. Then $k[G]$ equals  $k[\Delta(\overline G)]$.
\end{theorem}

\begin{theorem}
The edge ring $k[G]$ has a 2-linear resolution if and only if $\overline G$ is chordal.
\end{theorem}

\begin{theorem}
The Stanley-Reisner ring $k[\Delta]$ has a 2-linear resolution if and only if $\Delta$ is a fat forest.
\end{theorem}

\section{Split graphs}
A graph $G=(V,E)$ is called split if $V$ is the disjoint union $V=C\sqcup S$ of a complete subgraph $C$ and a stable set $S$.
It is easy to see that split graphs are chordal and that the complement of a split graph is a split graph. Let $C=\{x_1,\ldots,x_n\}$
and $S=\{y_1,\ldots,y_m\}$. For each $x_i\in C$ let $N_G(x_i)\cap S$ denote the neighbors of $x_i$ in $S$, and let $n_i=|N_G(x_i)\cap S|$.
Let $\Sigma$ be the simplicial complex such that $k[\Sigma]=k[G]$, so $\Sigma=\Delta(\overline G)$..

\begin{lemma}
If $n_i>0$ for each $i$, then $\Sigma$ has $n+1$ facets, $F_0$, the simplex on $S$, and $F_i$,
the simplex on= $\{x_i\}\cup S_i$, $i=1,\ldots,n$, where $S_i=S\setminus(N_G(x_i)\cap S)$.
If some $n_i=0$, say $n_1=0$, then $\Sigma$ has $n$ facets, $F_1$, the simplex on $\{x_1\}\cup S$, and $F_i$ as above, $i=2,\ldots,n$.
\end{lemma}

\begin{proof}
We have that $\Sigma=\Delta(\overline G)$, the largest simplicial complex with 1-simplex equal to $\overline G$.
\end{proof}

\begin{theorem}
Suppose $G$ is a split graph with $|N_G(x_i)\cap S|=n_i$, $i=1,\ldots,n$. Then the Hilbert series of $k[G]=k[\Sigma]$ is
$$\frac{1}{(1-t)^m}+\sum_{i=1}^n(\frac{1}{(1-t)^{m-n_i-1}}-\frac{1}{(1-t)^{m-n_i}}).$$
\end{theorem}

\begin{proof}
Suppose first that $n_i>0$ for all $i$. Then $\Sigma$ is a fat forest with $F_i$ as in the lemma. We have $(F_0\cup\cdots\cup F_{k-1})\cap F_k=F_0\cap F_k$.
Now $|F_0|=m$, $|F_k|=m-n_i+1$, and $|F_k\cap F_0|=m-n_i$ if $k>0$. This proves the formula in case all $x_i$ has a neighbor in $S$.
Now suppose $n_1=0$. Then $|F_1|=m+1$, $|F_k|=m-n_i+1$, and
$|(F_1\cup\cdots\cup F_{k-1})\cap F_k|=|F_1\cap F_k|=m-n_i$. This gives the formula 
$$\frac{1}{(1-t)^{m+1}}+\sum_{i=2}^n(\frac{1}{(1-t)^{m-n_i+1}}-\frac{1}{(1-t)^{m-n_i}}).$$
Now compare with the first formula. It gives
$$\frac{1}{(1-t)^m}+(\frac{1}{(1-t)^{m+1}}-\frac{1}{(1-t)^m})+\sum_{i=2}^n(\frac{1}{(1-t)^{m-n_i+1}}-\frac{1}{(1-t)^{m-n_i}}).$$
Since $\frac{1}{(1-t)^m}+(\frac{1}{(1-t)^{m+1}}-\frac{1}{(1-t)^m})=\frac{1}{(1-t)^{m+1}}$, we see that the first formula is true also here.
\end{proof}

\begin{theorem}
All split graphs have 2-linear resolution. The nonzero Betti numbers for $j>0$ are 
$$\beta_{j,j+1}(k[G])=(-1)^j(\binom{n}{j+1}+1+\sum_{i=1}^n(\binom{n+n_i+1}{j+1}-\binom{n+n_i}{j+1})).$$
\end{theorem}

\begin{proof}
The Hilbert series is 
$$\frac{(1-t)^n+\sum_{i=1}^n((1-t)^{n+n_i-1}-(1-t)^{n+n_i})}{(1-t)^{m+n}}.$$
Now $\beta_{j,j+1}$ equals $(-1)^j$ times the coefficient of $t^{j+1}$ in the numerator, so 
$$\beta_{j,j+1}=(-1)^j(\binom{n}{j+1}+\sum_{i=1}^n\binom{n+n_i-1}{j+1}-\binom{n+n_i}{j+1}).$$
\end{proof}

The Alexander dual $\Sigma^\vee$ of a simplicial complex $\Sigma$ on $\{1,\ldots,n\}$ is $\{F;F^c\notin\Sigma\}$, where $F^c=\{1,\ldots,n\}\setminus F\}$.

\begin{corollary}
If $k[\Sigma]$ is the Stanley-Reisner ring of a split graph, then $k[\Sigma^\vee]$ is Cohen-Macaulay with projective dimension 2.
\end{corollary}

\begin{proof}
Eagon-Reiner has shown \cite{E-R} that $k[\Sigma]$ is Cohen-Macaulay if and only if $k[\Sigma]$ has a linear resolution, and
Terai \cite{T} has shown that pd$(k[G^\vee])$ equals  $\max\{j-i;{\rm Tor}_{i,j}^A(I,k)\}=1+\max\{j-i;{\rm Tor}_{i,j}^A(k[\Sigma],k)$, 
where $A=k[x_1,\ldots,x_n]$ and $k[\Sigma]=A/I$.
\end{proof}

Now we give some examples of split graphs and their Betti numbers.

\begin{itemize}

\item{\bf Complete split graphs}

Let $G=C\sqcup S$ with $C=\{x_1,\ldots,x_n\}$, $S=\{y_1,\ldots,y_m\}$, and $(x_i,y_j)\in E$ for all $i$ and $j$. Then
$k[G]=k[\Sigma]$, where $\Sigma$ has facets $\mathcal S$, the simplex on $S$, and $\{x_i\}$, $i=1,\ldots,n$. The Hilbert series is 
$$\frac{1}{(1-t)^m}+\frac{n}{(1-t)}-n=\frac{(1-t)^n+n(1-t)^{n+m-1}-n(1-t)^{n+m}}{(1-t)^{n+m}}.$$ 
Thus, for $i>0$, the only nonzero Betti numbers are
$$\beta_{i,i+1}(k[G])=(-1)^i(\binom{n}{i+1}+n\binom{n+m+1}{i+1}-n\binom{n+m}{i+1}).$$
Singh-Verma \cite{S-V} get 
$$\beta_{i,i+1}=i\binom{n}{i+1}+\sum_{\substack{r+s=i+1\\r,s\ge1}}r\binom{n}{r}\binom{m}{s}.$$

\item{\bf Nearly complete split graphs}

Here we have the same edges as above except $(x_i,y_i)$, $i=1,\ldots,\min\{m,n\}$. We have the facets on $\mathcal S$, and
the simplex on $\{x_i,y_i\})$, $i=1,\ldots,N$.
The Hilbert series is
$$\frac{1}{(1-t)^{m}}+\frac{N}{(1-t)^2}-\frac{N}{(1-t)}=\frac{(1-t)^{n}+N(1-t)^{n+m-2}-N(1-t)^{n+m-1}}{(1-t)^{n+m}}.$$
Thus, for $i>0$, the only nonzero Betti numbers are
$$\beta_{i,i+1}(k[G])=(-1)^i(\binom{n}{i+1}+\binom{n+m-2}{i+1}-N\binom{n+m-1}{i+1}).$$
Singh-Verma \cite{S-V} get 
$$\beta_{i,i+1}=i\binom{n}{i+1}+\sum_{p=1}^ip\binom{n}{p}\binom{m-p}{i+1-p}+
\sum_{t=1}^N\binom{N}{t}\sum_{p=1}^{i+1-2t}p\binom{n-t}{p}\binom{m-t-p}{i+1-2t-p}.$$

\item{\bf Empty split graphs}

Here we mean that there are no edges between the two parts of vertices. The facets are $\mathcal S\cup\{x_i\}$, $i=1,\ldots,N$.
The Hilbert series is
$$\frac{N}{(1-t)^{m+1}}-\frac{N-1}{(1-t)^m}=\frac{N(1-t)^{n-1}-(N-1)(1-t)^{n}}{(1-t)^{n+m}}.$$
Thus, for $i>0$, the only nonzero Betti numbers are
$$\beta_{i,i+1}(k[G])=(-1)^i(N\binom{n-1}{i+1}+(N-1)\binom{n}{i+1}.$$

\item{\bf Every $x_i$ has exactly one neighbor in $S$}

Let the edges between the two parts be $(x_iy_{j_i})$, $i=1,\ldots,n$. The facets are $\mathcal S$ and $\{x_i\}\cup(S\setminus y_{j_i})$,
so the Hilbert series is 
$$\frac{n+1}{(1-t)^n}-\frac{n}{(1-t)^{n-1}}=\frac{(n+1)(1-t)^m-n(1-t)^{m+1}}{(1-t)^{m+n}}.$$
Thus, for $i>0$, the only Betti numbers are
$$\beta_{i,i+1}(k[G])=(-1)^i((n+1)\binom{m}{i+1}-n\binom{m+1}{i+1}).$$
\end{itemize}

\section{Dimension, projective dimension, depth, and Cohen-Macaulayness}
Now we will determine which edge rings of split graphs that are Cohen-Macaulay, i.e., have the same depth as dimension.
The dimension of $k[\Sigma]$ equals the largest size of a facet. Thus, if some $x_i$ has no neighbor in $S$, the dimension is $n+1$, and
otherwise it is $n$. That the depth of $k[G]$ is $n+1$ is equivalent to that the projective dimension is $m-1$ according to Auslander-Buchsbaum.
The formula for Betti numbers shows that pd$(k[G])=m+1$ is true only if all $x_i$ have no neighbor in $S$. In case the depth of $k[G]$ is $n$,
so pd$(k[G])=m$, we see that this corresponds to the case when all $x_i$ has exactly one neighbor in $S$. We conclude:

\begin{theorem}
For the split graph $G$, we get that $k[G]$ is Cohen-Macaulay if and only if either no $x_i$ has a neighbor in $S$, or all $x_i$ has exactly one
neighbor in $S$.
\end{theorem}

\begin{note}
This corresponds to the two last examples in the previous section.
\end{note}

\begin{note}
There is no conflict of interests.
\end{note}


\begin{thebibliography}{99}
 \bibitem{E-R}
J. A. Eagon and V. Reiner,
\emph{Resolutions of Stanley--Reisner rings and Alexander duality},
J. Pure Appl. Algebra \textbf{130} (1998), no. 3, 265--275.

\bibitem{F-H}
S. F\"oldes and P. L. Hammer,
\emph{Split graphs},
Proceedings of the Eighth Southeastern Conference on Combinatorics, Graph Theory and Computing, Congressus Numerantium \textbf{XIX} (1977), 311--315.

\bibitem{Fr}
R. Fr\"oberg,
\emph{On Stanley--Reisner rings},
Banach Center Publ. \textbf{26} (1990), no. 2, 57--70.

\bibitem{Fr1}
R. Fr\"oberg,
\emph{Betti numbers of fat forests and their Alexander dual},
J. Algebraic Combin. \textbf{56} (2022), 1023--1030.

\bibitem{S-V}
P. Singh and R. Verma,
\emph{Betti numbers of edge ideals of some split graphs},
Comm. Algebra \textbf{48} (2020), no. 12, 5026--5037, doi:10.1080/00927872.2020.1777559.

\bibitem{T}
N. Terai,
\emph{Alexander duality theorem and Stanley--Reisner rings},
S\=urikaisekikenky\=usho K\=oky\=uroku \textbf{1078} (1999), 174--184.


 \end{thebibliography}
\end{document}